\documentclass[11pt]{amsart}
\usepackage{amsthm}
\usepackage{amssymb}
\usepackage{amsmath}
\usepackage{tikz}
\usepackage{url}
\usepackage{cite}
\nocite{*}
\usepackage{a4wide}
\usepackage{subcaption}
\newtheorem{thm}{Theorem}
\newtheorem{prop}[thm]{Proposition}
\newtheorem{lem}[thm]{Lemma}

\theoremstyle{remark}
\newtheorem{rem}{Remark}

\newcommand{\comment}[1]{}

\newcommand{\mx}{\mathbf{x}}
\newcommand{\mI}{\mathbf{1}}
\newcommand{\CL}{\mathcal{L}}
\newcommand{\ad}{\operatorname{ad}}
\newcommand{\TW}{\operatorname{TW}}
\newcommand{\Pc}{\mathcal{P}}
\newcommand{\sgn}{\operatorname{sgn}}
\newcommand{\f}{\boldsymbol{f}}
\newcommand{\g}{\boldsymbol{g}}
\begin{document}

\thispagestyle{plain}

\title{The ancestral matrix of a rooted tree}

\author{Eric O. D. Andriantiana}
\address{Eric O. D. Andriantiana\\
Department of Mathematics (Pure and Applied)\\
Rhodes University, PO Box 94\\
6140 Grahamstown\\
South Africa}
\email{E.Andriantiana@ru.ac.za}

\author{Kenneth Dadedzi}
\address{Kenneth Dadedzi\\
Department of Mathematical Sciences\\
Stellenbosch University\\
Private Bag X1\\
Matieland 7602\\
South Africa}
\email{dadedzi@sun.ac.za}

\author{Stephan Wagner}
\address{Stephan Wagner\\
Department of Mathematical Sciences\\
Stellenbosch University\\
Private Bag X1\\
Matieland 7602\\
South Africa}
\email{swagner@sun.ac.za}

\thanks{This work was supported by the National Research Foundation of South Africa (grants 96236 and 96310).}


\maketitle

\begin{abstract}
Given a rooted tree $T$ with leaves $v_1,v_2,\ldots,v_n$, we define the ancestral matrix $C(T)$ of $T$ to be the $n \times n$ matrix for which the entry in the $i$-th row, $j$-th column is the level (distance from the root) of the first common ancestor of $v_i$ and $v_j$. We study properties of this matrix, in particular regarding its spectrum: we obtain several upper and lower bounds for the eigenvalues in terms of other tree parameters. We also find a combinatorial interpretation for the coefficients of the characteristic polynomial of $C(T)$, and show that for $d$-ary trees, a specific value of the characteristic polynomial is independent of the precise shape of the tree.
\end{abstract}

\section{Introduction}

In this paper, we will be interested in the combinatorial and spectral properties of a matrix associated with a rooted tree. A rooted tree has a distinguished vertex, called the root; the vertices of a rooted tree can be arranged in levels by their distance to the root: level $\ell$ consists of all vertices whose distance from the root is $\ell$. Thus the root is the only vertex at level $0$, and all the children (if there are any) of a level-$\ell$ vertex are at level $\ell+1$. A vertex without children will be called a leaf; this includes the root if it is the only vertex, but not otherwise. The set of leaves of a (rooted) tree $T$ will be denoted by $\CL(T)$, the number of leaves by $L(T)$.

Rooted trees occur naturally in many different areas, from data structures to phylogenetics. This work is an attempt to introduce the powerful framework of spectral graph theory to the world of rooted trees by studying what will be called the ancestral matrix of a rooted tree. In order to define it formally, we need a few ingredients. A rooted tree can be regarded as the Hasse diagram of a poset, where the root is the greatest (or least if we reverse the order) element. For any two elements $v,w$ of this poset, there is a unique supremum $v \vee w$, the least element that is simultaneously greater than or equal to both $v$ and $w$. In terms of the tree structure, this can be interpreted as the lowest element (farthest from the root) that is an ancestor of both $v$ and $w$. The ancestral level of $v$ and $w$ is the level of $v \vee w$, i.e., the greatest distance of a common ancestor from the root. We will denote it by $\ell(v \vee w)$. It is worth pointing out a connection between the ancestral level and distances: if $r$ denotes the root and $d(\cdot,\cdot)$ the usual graph distance, then we have
\begin{equation}\label{eq:dist_level}
d(v,w) = d(v,r) + d(w,r) - 2 \ell(v \vee w),
\end{equation}
since the path from $v$ to $r$ and the path from $w$ to $r$ both include the path from $v \vee w$ to $r$, while the remaining parts form the path from $v$ to $w$.

The ancestral level is a way to measure how close two vertices are. For example, if we interpret the rooted tree as a phylogenetic tree (see \cite{SempleSteel}), then it represents the point at which two species are separated. In an important data structure known as a trie (see e.g. \cite[Section 6.3]{Knuth}), where the leaves store data according to certain keys (strings over a given alphabet), the ancestral level is the length of the longest common prefix.

To define the ancestral matrix, we focus on the leaves. Let $v_1,v_2,\ldots,v_n$ be the leaves of a rooted tree $T$. The ancestral matrix $C(T)$ is defined by its entries $c_{ij}$ in the following way:
$$c_{ij} = \ell(v_i \vee v_j).$$
For the example in Figure~\ref{fig:first_example}, the ancestral matrix is
$$\begin{bmatrix} 2 & 1 & 0 & 0 & 0 & 0 \\ 1 & 2 & 0 & 0 & 0 & 0 \\ 0 & 0 & 2 & 1 & 1 & 1 \\ 0 & 0 & 1 & 2 & 1 & 1 \\ 0 & 0 & 1 & 1 & 3 & 2 \\  0 & 0 & 1 & 1 & 2 & 3 \end{bmatrix}.$$
The ancestral matrix is similar in nature to a meet matrix, see \cite{Haukkanen1996meet}. Meet matrices can be defined on arbitrary posets; their determinants are particularly well-studied. 
Some basic properties of the ancestral matrix are immediate: it is clearly always a symmetric matrix, and the diagonal entry is always the unique maximum in each row and column.

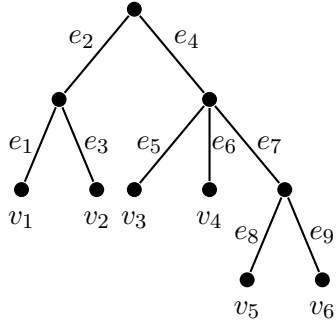
\begin{figure}[htbp]\centering  
\begin{tikzpicture}[thick, level distance=12mm, 
  every node/.style = {shape=circle, fill, inner sep=2pt}]
\tikzstyle{level 1}=[sibling distance=20mm]
\tikzstyle{level 2}=[sibling distance=10mm]
\tikzstyle{level 2}=[sibling distance=10mm]
\node {}
child {node {}
child {node {}} 
child {node {}}}
child {node {} 
child {node {}} 
child {node {}}
child {node {}
child {node{}}
child {node{}}
}};
\node [draw=none,fill=none] at (-1.5,-2.8) {$v_1$};
\node [draw=none,fill=none] at (-0.5,-2.8) {$v_2$};
\node [draw=none,fill=none] at (0,-2.8) {$v_3$};
\node [draw=none,fill=none] at (1,-2.8) {$v_4$};
\node [draw=none,fill=none] at (1.5,-4) {$v_5$};
\node [draw=none,fill=none] at (2.5,-4) {$v_6$};

\node [draw=none,fill=none] at (-1.5,-1.8) {$e_1$};
\node [draw=none,fill=none] at (-0.7,-0.4) {$e_2$};
\node [draw=none,fill=none] at (-0.5,-1.8) {$e_3$};
\node [draw=none,fill=none] at (0.7,-0.4) {$e_4$};
\node [draw=none,fill=none] at (0.2,-1.8) {$e_5$};
\node [draw=none,fill=none] at (1.2,-1.8) {$e_6$};
\node [draw=none,fill=none] at (1.8,-1.8) {$e_7$};
\node [draw=none,fill=none] at (1.5,-3) {$e_8$};
\node [draw=none,fill=none] at (2.5,-3) {$e_9$};

\end{tikzpicture}
\caption{Example of a rooted tree.}\label{fig:first_example}
\end{figure}

Another important structural property relates to the branches of a rooted tree: let $T_1$, $T_2$, \ldots, $T_k$ be the branches of a rooted tree $T$ (i.e., the connected components that remain when the root of $T$ is removed, each endowed with its natural root). For every $j \in \{1,2,\ldots,k\}$, let $E_j$ be a square matrix whose entries are all equal to $1$, and whose size (number of rows) equals the number of leaves of $T_j$. Then the ancestral matrix $C(T)$ has the following block diagonal form with respect to a suitable order of leaves:
\begin{equation}\label{eq:block_diag}
C(T) = \begin{bmatrix} C(T_1) + E_1 & 0 & 0 & \cdots & 0 \\ 0 & C(T_2) + E_2 & 0 & \cdots & 0 \\ 0 & 0 & C(T_3) + E_3 & \cdots & 0 \\ \vdots & \vdots & \vdots & \ddots & \vdots \\ 0 & 0 & 0 & \cdots & C(T_k) + E_k \end{bmatrix}.
\end{equation}
This is because the ancestral level of two leaves that lie in distinct branches is always $0$, while the ancestral level of two leaves in the same branch increases by $1$ in $T$.

The ancestral spectrum of a rooted tree $T$ is now defined in analogy to other graph spectra (such as the adjacency spectrum or the Laplacian spectrum) as the spectrum of the ancestral matrix $C(T)$. Note that this spectrum does not depend on the order of leaves, and that all eigenvalues are necessarily real since $C(T)$ is symmetric. In the example of Figure~\ref{fig:first_example}, the eigenvalues are (in decreasing order) $4+\sqrt{5},3,4-\sqrt{5},1,1,1$.

In analogy with the fact that the Laplacian and the signless Laplacian of a graph can be obtained as the product of an incidence matrix with its own transpose, a similar identity holds for the ancestral matrix. To this end, we define a path incidence matrix $I_p(T)$ of a rooted tree $T$. Let $P(u,v)$ be the set of edges on the path from vertex $u$ to vertex $v$ in a tree $T$. Suppose $v_1, v_2,\ldots , v_n$  and $e_1,e_2, \ldots , e_m$ are the leaves and the edges of a rooted tree $T$ with root $r$. The path incidence matrix $I_p(T)$ is defined as an $n\times m$ matrix whose entries are 
\[a_{ij} = \begin{cases} 1 & \text{if } e_j \in P(v_i,r), \\ 0 & \text{otherwise}. \end{cases} \]

The path incidence matrix for the rooted tree in Figure~\ref{fig:first_example} is
$$\begin{bmatrix} 
1 & 1 & 0 & 0 & 0 & 0 & 0 & 0 & 0 \\ 
0 & 1 & 1 & 0 & 0 & 0 & 0 & 0 & 0 \\
0 & 0 & 0 & 1 & 1 & 0 & 0 & 0 & 0 \\ 
0 & 0 & 0 & 1 & 0 & 1 & 0 & 0 & 0 \\ 
0 & 0 & 0 & 1 & 0 & 0 & 1 & 1 & 0  \\  
0 & 0 & 0 & 1 & 0 & 0 & 1 & 0 & 1 \end{bmatrix}.$$

Note that the row sums are equal to the respective depths of the leaves and the column sums provide information on the number of leaves ``below'' a certain edge. It is easy to see that for a rooted tree $T$, we have 
$$C(T) =  I_p(T)I_p(T)^t.$$
An immediate consequence of this identity is the fact that the ancestral matrix is positive semidefinite (in fact positive definite, as we will see in the next section). 

\section{The eigenvalues of the ancestral matrix}

This section will be devoted to the eigenvalues of the ancestral matrix of a rooted tree. The Rayleigh quotient will play an important role in this context: for a symmetric matrix $A$ and a vector $\mathbf{x}$, it is given by
$$R(A,\mx) = \frac{\mx^t A \mx}{\mx^t \mx}.$$
It is well known that $R(A,\mx) = \lambda$ if $\mx$ is an eigenvector for the eigenvalue $\lambda$, and that the greatest and least eigenvalue are given by
\begin{equation}\label{eq:rayleigh_max}
\sup_{\mx \neq \mathbf{0}} R(A,\mx) = \sup_{\|\mx\| = 1} R(A,\mx) = \sup_{\|\mx\| = 1} \mx^t A \mx
\end{equation}
and
\begin{equation}\label{eq:rayleigh_min}
\inf_{\mx \neq \mathbf{0}} R(A,\mx) = \inf_{\|\mx\| = 1} R(A,\mx) = \inf_{\|\mx\| = 1} \mx^t A \mx
\end{equation}
respectively.

We start our considerations with a lower bound on the eigenvalues. As it turns out, the eigenvalue $1$ plays a specific role, which is captured in the following theorem.

\begin{thm}
If $T$ is a rooted tree that does not only consist of the root, then all eigenvalues are greater than or equal to $1$. Moreover, the multiplicity of $1$ as an eigenvalue of $C(T)$ is given by
$$\big(\text{\textnormal{number of leaves of $T$}} \big) - \big( \text{\textnormal{number of non-root vertices of $T$ adjacent to a leaf}} \big).$$
A basis for the eigenspace of the eigenvalue $1$ is obtained in the following way: for every maximal (with respect to inclusion) $r$-tuple $w_1,w_2,\ldots,w_r$ of leaves that share a common parent, take all vectors with an entry $1$ in the row corresponding to $w_1$, an entry $-1$ in the row corresponding to $w_j$ for some $j \in \{2,3,\ldots,r\}$, and otherwise zeros. Moreover, if there is at least one leaf adjacent to the root, pick one such leaf $u$ and take the vector with an entry $1$ in the row corresponding to $u$ and otherwise zeros. The set of all these vectors is a basis for the eigenspace of $1$.
\end{thm}

\begin{proof}
We prove the statement by induction on the number of vertices. The statement is vacuously true if $T$ has only one vertex (the root), so we consider the situation that $T$ has one or more branches, denoted by $T_1,T_2,\ldots,T_k$. It follows from the block diagonal representation of $C(T)$ in~\eqref{eq:block_diag} that the spectrum of $C(T)$ is the union of the spectra of $C(T_1) + E_1$, $C(T_2) + E_2$, etc. We consider two cases for a branch $T_j$:
\begin{itemize}
\item If $T_j$ only consists of one vertex, then $C(T_j) + E_j$ is a $1 \times 1$-matrix whose only entry is $1$. This yields an eigenvalue $1$.
\item If $T_j$ has more than one vertex, then we already know that
$$\inf_{\mx \neq \mathbf{0}} R(C(T_j),\mx) \geq 1$$
by the induction hypothesis. Moreover, since $E_j$ is positive semidefinite, we have $R(E_j,\mx) \geq 0$ with equality if and only if $\mx$ is orthogonal to the all-$1$ vector $\mI$, or equivalently if the sum of all entries of $\mx$ is $0$. Thus
\begin{equation}\label{eq:induction_step}
\inf_{\mx \neq \mathbf{0}} R(C(T_j) + E_j,\mx) = \inf_{\mx \neq \mathbf{0}} \big( R(C(T_j),\mx) + R(E_j,\mx) \big) \geq \inf_{\mx \neq \mathbf{0}} R(C(T_j),\mx) \geq 1
\end{equation}
by the induction hypothesis.
\end{itemize}
It follows that every eigenvalue is greater than or equal to $1$, which proves the first assertion. Now let us look at the associated eigenvectors: for each single-vertex branch, we have a unit eigenvector whose only non-zero entry corresponds to the single vertex of the branch. If $u_1,u_2,\ldots,u_s$ are (without loss of generality) all leaves adjacent to the root, then each of them gives rise to such a unit eigenvector, and these $s$ eigenvectors are clearly linearly independent. However, we can replace them by a different set of vectors that spans the same space: the unit vector corresponding to $u_1$, and for every $j > 1$ the vector with an entry $1$ corresponding to $u_1$, an entry $-1$ corresponding to $u_j$, and otherwise zeros. This agrees with our description of eigenvectors.

For each branch $T_j$ that is not a single vertex, we need eigenvectors that satisfy~\eqref{eq:induction_step} with equality. For this purpose, $\mx$ needs to be an eigenvector of $C(T_j)$ with respect to the eigenvalue $1$, and $\mx$ needs to be orthogonal to the all-$1$ vector $\mI$. By the induction hypothesis, we have a basis for the eigenspace of $1$ as an eigenvalue of $C(T_j)$, and it is clear that those eigenvectors with an entry $1$ and an entry $-1$ form a basis of the subspace that is orthogonal to the all-$1$ vector $\mI$. Thus these remain eigenvectors for $C(T)$ (when suitably padded with zeros) and form a basis for the eigenspace of $1$ as an eigenvalue of $C(T_j) + E_j$. Each maximal $r$-tuple of leaves with a common parent vertex thus contributes $r-1$ to the multiplicity of $1$ as an eigenvalue, unless the common parent is the root, in which case the contribution is $r$. The formula for the multiplicity follows immediately.
\end{proof}

So we know now in particular that the eigenvalues of $C(T)$ are not only real and non-negative, but even positive, unless $T$ only has a single vertex. Next we look at the maximum eigenvalue of $C(T)$, i.e.~the spectral radius, which we denote by $\rho_C(T)$ and call the ancestral spectral radius. Making use of the block diagonal shape once again, we see that the spectral radius is the maximum of the spectral radii of the matrices $C(T_1) + E_1$, $C(T_2) + E_2$, etc. By the Perron-Frobenius Theorem, the multiplicity of $\rho_C(T)$ as an eigenvalue of $C(T_j) + E_j$ is at most $1$, and if $\rho_C(T)$ is an eigenvalue of $C(T_j) + E_j$, then there exists an eigenvector with positive entries for it. Hence the multiplicity of $\rho_C(T)$ is less than or equal to the root degree/number of root branches (equality can hold, e.g. if all root branches are isomorphic). In the following, we will call any eigenvector associated with $\rho_C(T)$ that has non-negative real entries a Perron vector of $T$.

The following proposition is analogous to the well-known fact that the spectral radius of the adjacency matrix of a graph lies between the average and the maximum degree (see for example~\cite[(1.5)]{Stevanovic2014spectral}).

\begin{prop}\label{prop:total_ancestral_bounds}
Let the total ancestral depth of a leaf $v$ in $T$ be defined by
$$\ad(v) = \sum_{w \in \CL(T)} \ell(v \vee w),$$
where the sum is over all leaves $w$. For every rooted tree $T$, we have
$$\frac{1}{L(T)} \sum_{v \in \CL(T)} \ad(v) \leq \rho_C(T) \leq \max_{v \in \CL(T)} \ad(v).$$
\end{prop}

\begin{proof}
Note that $\ad(v)$ is precisely the row sum of the row that corresponds to $v$. For the lower bound, we consider the Rayleigh quotient of the vector $\mI$. Since $\mI^t \mI = L(T)$ and $\mI^t C(T) \mI$ is the sum of all entries of $C(T)$, we have
$$\rho_C(T) \geq R(C(T),\mI) = \frac{1}{L(T)} \sum_{v \in \CL(T)} \ad(v).$$
For the upper bound, consider an eigenvector $\mx$ associated with $\rho_C(T)$, and denote the entry of $\mx$ associated with vertex $v$ by $x(v)$. Recall that the eigenvector can be chosen to have only non-negative entries. Let $w$ be the vertex for which $x(w)$ attains its maximum value. The eigenvalue equation gives us
$$\rho_C(T) x(w) = \sum_v \ell(v \vee w) x(v) \leq \sum_v \ell(v \vee w) x(w) = \ad(w) x(w),$$
thus
$$\rho_C(T) \leq \ad(w),$$
from which the upper bound follows immediately.
\end{proof}

By means of the identity~\eqref{eq:dist_level}, $\ad(v)$ can be rewritten in terms of distances. Specifically, we have
\begin{align*}
\ad(v) &= \frac12 \sum_{w \in \CL(T)} (d(v,r) + d(w,r) - d(v,w)) \\
&= \frac12 \big( L(T) d(v,r) + D_T(r) - D_T(v) \big),
\end{align*}
where $D_T(v)$ is the sum of all distances from $v$ to the leaves of $T$. Summing over all leaves $v$, we get
$$\sum_{v \in \CL(T)} \ad(v) = \frac12 \Big( 2L(T) D_T(r) - \sum_{v \in \CL(T)} D_T(v) \Big) = L(T) D_T(r) - \TW(T),$$
where $\TW(T)$ represents the terminal Wiener index, i.e.~the sum of all distances between pairs of leaves:
$$\TW(T) = \sum_{\{v,w\} \subseteq \CL(T)} d(v,w).$$
Hence the lower bound in Proposition~\ref{prop:total_ancestral_bounds} becomes
\begin{equation}\label{eq:lower_bound_tw}
D_T(r) - \frac{\TW(T)}{L(T)} \leq \rho_C(T).
\end{equation}
Another simple lower bound for the spectral radius $\rho_C(T)$ is given by the height $h(T)$, i.e.~the greatest distance of a leaf from the root.

\begin{prop}\label{prop:height}
For every rooted tree $T$, we have $\rho_C(T) \geq h(T)$.
\end{prop}

\begin{proof}
Let $v$ be a leaf whose distance to the root equals $h(T)$, and take $\mx$ to be the unit vector with one entry $1$ corresponding to $v$ and otherwise only zeros. It is easy to see that $R(C(T),\mx) = h(T)$, so the statement follows immediately from~\eqref{eq:rayleigh_max}.
\end{proof}

The inequality in Proposition~\ref{prop:height} is actually sharp for every value of $h(T)$ and every value of $L(T)$: to see this, consider a tree consisting of the root, $n-1$ leaves attached to the root, and a path of length $h$ attached to the root (at one of its ends).

For the star $S_n$ (consisting only of a root and $n$ leaves attached to it), we have $\rho_C(S_n) = 1$ for every $n$. Thus the trivial bound $\rho_C(T) \geq 1$ is in fact sharp for all possible sizes of $T$. However, the lower bound can be improved if the degrees are restricted, as is shown in the following theorem:

\begin{figure}[htbp]\centering  
\begin{tikzpicture}[thick, level distance=12mm, 
  every node/.style = {shape=circle, fill, inner sep=2pt}]
\tikzstyle{level 1}=[sibling distance=30mm]
\tikzstyle{level 2}=[sibling distance=10mm]
\node {}
child {node {}
child {node {}} 
child {node {}}
child {node {}}}
child {node {}
child {node {}} 
child {node {}}
child {node {}}}
child {node {}
child {node {}} 
child {node {}}
child {node {}}}
;

\end{tikzpicture}
\caption{A complete ternary tree.}\label{fig:complete_ternary}
\end{figure}
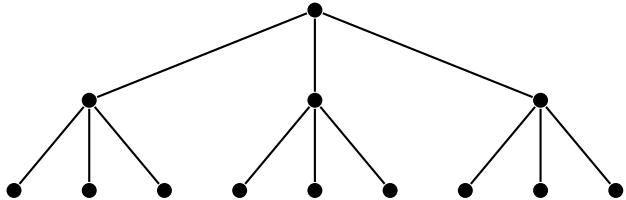

\begin{thm}
Let $T$ be a rooted tree for which the outdegree (number of children) of all vertices is less than or equal to $\Delta$. Then we have
$$\rho_C(T) \geq \frac{L(T)-1}{\Delta-1}.$$
Equality holds if and only if $T$ is a complete $\Delta$-ary tree, i.e.~a rooted tree for which all leaves lie on the same level and all internal vertices have precisely $\Delta$ children; see Figure~\ref{fig:complete_ternary} for an example in the case $\Delta = 3$.
\end{thm}

\begin{proof}
We make use of the lower bound of Proposition~\ref{prop:total_ancestral_bounds}. It will be shown that
\begin{equation}\label{eq:Q_inv_lower}
\sum_{v \in \CL(T)} \ad(v) \geq \frac{L(T)(L(T)-1)}{\Delta-1},
\end{equation}
from which the stated inequality follows. Let us use the shorthand 
$$Q(T) = \sum_{v \in \CL(T)} \ad(v) = \sum_{v \in \CL(T)} \sum_{w \in \CL(T)} \ell(v \vee w).$$
Our first goal is a recursion for this quantity in terms of the branches of $T$. Let these branches be denoted by $T_1,T_2,\ldots,T_k$. For two leaves $v,w$ in distinct branches, we simply have $\ell(v \vee w) = 0$. For two leaves $v,w$ in the same branch $T_i$, $\ell(v \vee w)$ increases by $1$ in $T$ compared to $T_i$. Thus we have
\begin{align*}
Q(T) &= \sum_{i=1}^k \Big( Q(T_i) + \sum_{v \in \CL(T_i)} \sum_{w \in \CL(T_i)} 1 \Big) \\
&= \sum_{i=1}^k \big( Q(T_i) + L(T_i)^2 \big).
\end{align*}
Now we prove~\eqref{eq:Q_inv_lower} by induction on the number of vertices of $T$. If there is only one vertex, then both sides of the inequality are $0$, so it holds. Otherwise, there are one or more branches $T_1,T_2,\ldots,T_k$ (where $k \leq \Delta$). The recursion for $Q(T)$, combined with the induction hypothesis, gives us
\begin{align*}
Q(T) &= \sum_{i=1}^k \big( Q(T_i) + L(T_i)^2 \big) \\
&\geq \sum_{i=1}^k  \Big( \frac{L(T_i)(L(T_i)-1)}{\Delta-1} + L(T_i)^2 \Big) \\
&= \frac{\Delta}{\Delta-1} \sum_{i=1}^k L(T_i)^2 - \frac{1}{\Delta-1} \sum_{i=1}^k L(T_i).
\end{align*}
We have $\sum_{i=1}^k L(T_i) = L(T)$, so the final term simplifies to $\frac{L(T)}{\Delta-1}$. Moreover, the inequality between the quadratic and the arithmetic mean gives us
$$\sum_{i=1}^k L(T_i)^2 \geq \frac{1}{k} \Big( \sum_{i=1}^k L(T_i) \Big)^2 = \frac{L(T)^2}{k} \geq \frac{L(T)^2}{\Delta}.$$
Putting everything together, we obtain
$$Q(T) \geq \frac{\Delta}{\Delta-1} \cdot \frac{L(T)^2}{\Delta} - \frac{1}{\Delta-1} \cdot L(T) = \frac{L(T)(L(T)-1)}{\Delta-1},$$
which completes the induction. Note that equality holds if and only if $k = \Delta$, $L(T_1)=L(T_2) = \cdots = L(T_k)$ and equality holds for each of the branches $T_i$. It is easy to deduce that equality holds if and only if $T$ is a complete $\Delta$-ary tree, as stated.
\end{proof}

We remark that the parameter $Q$ is somewhat similar in its definition to the Wiener index (sum of distances between all pairs of vertices), which is known to be minimised by complete $\Delta$-ary trees as well, see \cite{Fischermann2002Wiener}.

Moving our attention to upper bounds, we focus on classes of trees with fixed parameters such as outdegree sequence, number of vertices and number of leaves. The outdegree of a vertex is the number of children, and the outdegree sequence is the sequence of outdegrees of all vertices in a rooted tree. As it turns out, the maximum of the ancestral spectral radius $\rho_C(T)$ is typically attained by a so-called caterpillar tree. A rooted caterpillar $T$ is a rooted tree with the property that removing all of its leaves yields a path with the root at one end. The resulting path is called the backbone or spine of the caterpillar $T$.

Before we state our results, we first introduce some useful lemmas. The main idea is to apply certain tree operations that affect the ancestral spectral radius while preserving some of the features of the tree (such as the outdegree sequence). 

Firstly, we introduce an operation that moves branches away from the root along a path while preserving the number of leaves. In this way, we increase the levels of the common ancestors of some of the leaves. Formally, this operation can be described as follows: let $v_1,v_2,\ldots,v_k$ be consecutive vertices (in this order) on a path from the root to a leaf of a rooted tree $T$, and assume that $v_k$ is not a leaf. Moreover, let $B$ be a branch attached to $v_1$ (a subtree consisting of a child of $v_1$ and all its descendants). We construct a tree $T^\prime$ by moving $B$ to $v_k$ (by removing the edge between $B$'s root and $v_1$ and replacing it with an edge to $v_k$). This is illustrated in Figure \ref{fig:BSO}. We shall call this operation the branch shift operation.  It turns out that this operation increases the ancestral spectral radius, which is captured in the following lemma.

\begin{figure}[htbp]
\begin{subfigure}[b]{0.4\textwidth}
\centering  
\begin{tikzpicture}[thick, level distance=10mm,scale=0.6]
\foreach \i in {2,4,6,8,10}{\draw (0,\i) node[shape=circle, fill, inner sep=2pt]{};}
\draw (0,10)--(0,8) (0,6)--(0,4);
\draw[dashed] (0,8)--(0,6) (0,4)--(0,2);
\draw(0,10)--(2,10.5)--(2,9.5)--(0,10)(0,8)--(2,8.5)--(2,7.5)--(0,8)(0,6)--(2,6.5)--(2,5.5)--(0,6)(0,6)--(-2,6)--(-4,6.5)--(-4,5.5)--(-2,6)(0,4)--(2,4.5)--(2,3.5)--(0,4)(0,2)--(-1,0)--(1,0)--(0,2) ;
\draw (0,10.5)node{$r$}(-3.5,6)node{$B$}(0.3,6.4)node{$v_1$}(1.5,6)node{$A_1$}(1.5,4)node{$A_2$}(-0.5,4)node{$v_2$}(0,0.9)node{$A_k$}(0.5,2)node{$v_k$};
\end{tikzpicture}
\caption{The rooted tree $T$.}
\end{subfigure}
\begin{subfigure}[b]{0.4\textwidth}
\centering
\begin{tikzpicture}[thick, level distance=10mm, scale=0.6]
\foreach \i in {2,4,6,8,10}{\draw (0,\i) node[shape=circle, fill, inner sep=2pt]{};}
\draw (0,10)--(0,8) (0,6)--(0,4);
\draw[dashed] (0,8)--(0,6) (0,4)--(0,2);
\draw (0,10)--(2,10.5)--(2,9.5)--(0,10)(0,8)--(2,8.5)--(2,7.5)--(0,8)(0,6)--(2,6.5)--(2,5.5)--(0,6)(0,2)--(-2,2)--(-4,2.5)--(-4,1.5)--(-2,2)(0,4)--(2,4.5)--(2,3.5)--(0,4)(0,2)--(-1,0)--(1,0)--(0,2) ;
\draw (0,10.5)node{$r$}(-3.5,2)node{$B$}(-0.5,6)node{$v_1$}(1.5,6)node{$A_1$}(1.5,4)node{$A_2$}(-0.5,4)node{$v_2$}(0,0.9)node{$A_k$}(0.5,2)node{$v_k$};
\end{tikzpicture}
\caption{The rooted tree $T^\prime$.}
\end{subfigure}
\caption{$T^\prime$ is obtained from $T$ by the branch shift operation.}\label{fig:BSO}
\end{figure}
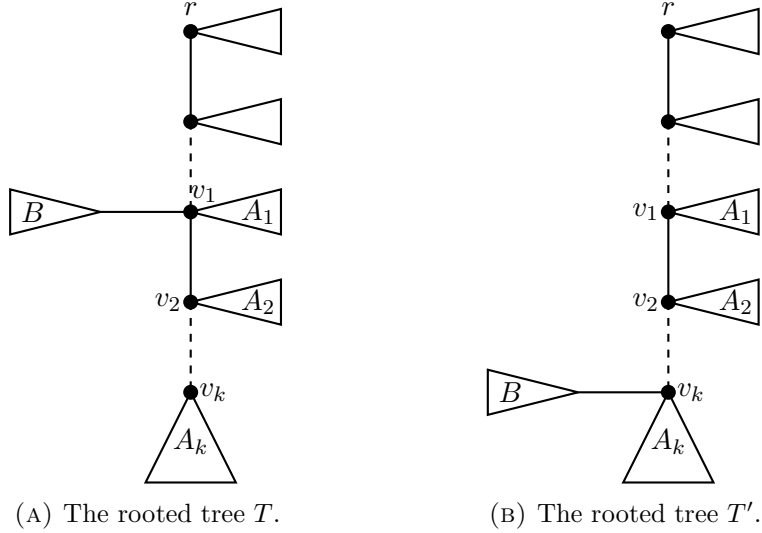

\begin{lem}\label{Lem:Branch_shifting}
Suppose that $T^\prime$ is obtained from a rooted tree $T$ by the branch shift operation as described above and depicted in Figure \ref{fig:BSO}. Assume moreover that there is a Perron vector for $T$ with the property that the entries corresponding to leaf descendants of $v_k$ (all leaves for which the path to the root passes through $v_k$) are positive. Then $\rho_C (T^\prime) > \rho_C  (T)$.
\end{lem}
\begin{proof}
As indicated in the figure, we denote the subtree consisting of $v_k$ and all its descendants by $A_k$, and we denote the subtrees consisting of $v_1,v_2,\ldots,v_{k-1}$ and all their respective descendants not lying on the path $v_1,v_2,\ldots,v_k$, the subtree $A_k$ or the branch $B$ by $A_1,A_2,\ldots,A_{k-1}$.

Let $\f$ be a unit Perron vector whose entries corresponding to the leaves in $A_k$ are positive; such a vector exists by assumption. We write $f(u)$ for the entry of $\f$ corresponding to a leaf $u$ in $T$. In the following, we will use $\ell_T$ to indicate the level of a vertex in $T$ (to emphasize the dependence on the tree).

The main idea of the proof of this Lemma is to show that the difference between the Rayleigh quotients $R(C(T^\prime),\f)$ and $R(C(T),\f)$ defined on the  Perron vector $\f$ of the rooted tree $T$ is strictly positive.

Note that $\ell_{T^\prime}(x\vee y)- \ell_{T}(x\vee y) = 0$ for all pairs of leaves $x,y$ that do not belong to $B$. Hence we can ignore all such pairs. The same is true  if $x$ lies in $B$, but $y$ does not lie in $\cup_{i=2}^k A_i$, or vice versa.

If $x$ lies in $B$ and $y$ in $A_i$ for some $i$, or vice versa, then we have
\[\ell_{T^\prime}(x\vee y) = (i-1) + \ell_{T}(x\vee y).\]
Finally, if $x$ and $y$ are both leaves in $B$, then
\begin{equation}\label{eq:both_in_B}
\ell_{T^\prime}(x\vee y) = (k-1) + \ell_{T}(x\vee y).
\end{equation}

Therefore, we can deduce that the entries of the matrix $C(T^\prime)$ are greater than or equal to those of the matrix $C(T)$, and it follows that
\begin{equation}\label{eq:TvsT'}
\rho_C (T^\prime) \geq R(C(T^\prime),\f) \geq R(C(T),\f) = \rho_C (T).
\end{equation}
In view of~\eqref{eq:both_in_B}, we even have
$$R(C(T^\prime),\f) \geq R(C(T),\f) + (k-1) \Big( \sum_{u \in B} f(u) \Big)^2,$$
so equality in~\eqref{eq:TvsT'} can only hold if $f(u) = 0$ for all $u \in B$. Moreover, for equality to hold, $\f$ would also have to be a Perron vector of $T^{\prime}$. But since $\f$ is non-zero on $A_k$ by assumption, it has to be non-zero on all leaves that belong to the same root branch of $T^{\prime}$ as $A_k$, in particular the leaves that belong to $B$. Thus we must have strict inequality. In fact, this argument even shows that $f(u) = 0$ is only possible in $T$ for leaves $u  \in B$ if $v_1$ is the root. Hence the statement of the lemma holds.
\end{proof}

Another important tree operation that increases the ancestral spectral radius of a rooted tree is the star shift operation. This operation increases the number of internal vertices by 1, but preserves the number of leaves in the tree. Let $v_1$ be a vertex of a rooted tree $T$ all of whose (at least two) children are leaves, and let $u$ be one of these leaves. The star shift operation introduces a new vertex $v_2$, which becomes a child of $v_1$. The leaf $u$ remains a child of $v_1$, while all other children of $v_1$ become children of $v_2$. The result is a tree $T^\prime$, see Figure \ref{fig:SSO} for an illustration.

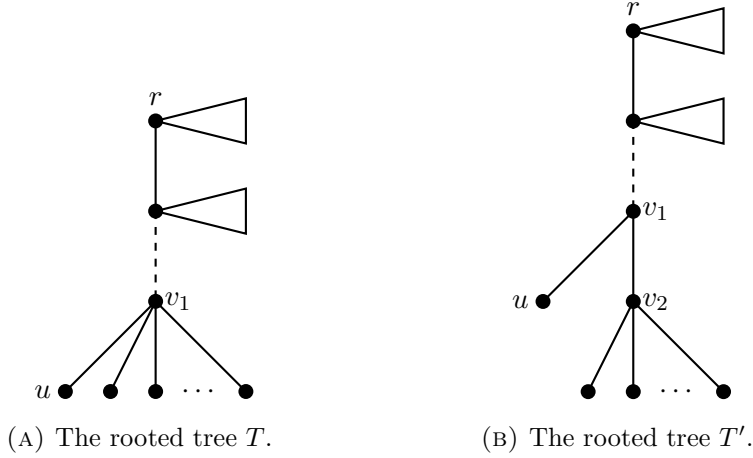
\begin{figure}[htbp]
\begin{subfigure}[b]{0.4\textwidth}
\centering  
\begin{tikzpicture}[thick, level distance=10mm,scale=0.6]
\foreach \i in {6,8,10}{\draw (0,\i) node[shape=circle, fill, inner sep=2pt]{};}
\foreach \i in {-2,-1,0,2}{\draw (\i,4) node[shape=circle, fill, inner sep=2pt]{};}
\draw (0,10)--(0,8)(0,6)--(0,4)(0,6)--(-2,4)(0,6)--(-1,4)(0,6)--(2,4);
\draw[dashed] (0,8)--(0,6);
\draw (0,10)--(2,10.5)--(2,9.5)--(0,10)(0,8)--(2,8.5)--(2,7.5)--(0,8);
\draw (0,10.5)node{$r$}(0.5,6)node{$v_1$}(1,4)node{$\cdots$}(-2.5,4)node{$u$};
\end{tikzpicture}
\caption{The rooted tree $T$.}
\end{subfigure}
\begin{subfigure}[b]{0.4\textwidth}
\centering
\begin{tikzpicture}[thick, level distance=10mm, scale=0.6]
\foreach \i in {6,8,10}{\draw (0,\i) node[shape=circle, fill, inner sep=2pt]{};}
\foreach \i in {-1,0,2}{\draw (\i,2) node[shape=circle, fill, inner sep=2pt]{};}
\foreach \i in {-2,0}{\draw (\i,4) node[shape=circle, fill, inner sep=2pt]{};}
\draw (0,10)--(0,8)(0,6)--(0,4)--(0,2)(0,6)--(-2,4)(0,4)--(-1,2)(0,4)--(2,2);
\draw[dashed] (0,8)--(0,6);
\draw (0,10)--(2,10.5)--(2,9.5)--(0,10)(0,8)--(2,8.5)--(2,7.5)--(0,8);
\draw (0,10.5)node{$r$}(0.5,6)node{$v_1$}(0.5,4)node{$v_2$}(1,2)node{$\cdots$}(-2.5,4)node{$u$};
\end{tikzpicture}
\caption{The rooted tree $T^\prime$.}
\end{subfigure}
\caption{$T^\prime$ is obtained from $T$ by the star shift operation.}\label{fig:SSO}
\end{figure}

\begin{lem}\label{Lem:Star_shift}
Suppose that $T^\prime$ is obtained from a rooted tree $T$ by the star shift operation as described above and depicted in Figure \ref{fig:SSO}. Assume moreover that there is a Perron vector for $T$ with the property that the entries corresponding to children of $v_1$ are positive. Then $\rho_C (T^\prime) > \rho_C  (T)$.
\end{lem}

\begin{proof}
We let $\f$ be a unit Perron vector with the property that the entries corresponding to children of $v_1$ are positive. Observe that
$$\ell_{T^\prime}(x\vee y)- \ell_{T}(x\vee y) = 0$$
unless both $x$ and $y$ are children of $v_2$ in $T^\prime$. In the latter case, we have
$$\ell_{T^\prime}(x\vee y)- \ell_{T}(x\vee y) = 1$$
So as in the previous lemma, the entries of the matrix $C(T^\prime)$ are greater than or equal to those of the matrix $C(T)$. This together with the assumption that the entries of $\f$ corresponding to the children of $v_2$ are positive shows that
\[\rho_C (T^\prime) \geq R(C(T^\prime),\f) > R(C(T),\f) = \rho_C (T).\]
\end{proof}

Finally, we introduce the leaf swap operation (LSO). Similar to the branch shift operation, it moves branches away from the root. The setup is depicted in Figure \ref{fig:LSO}. An important feature of this operation is that it does not change the outdegree sequence. As in the setup of the branch shift operation, we let $v_1,v_2,\ldots,v_k$ be consecutive vertices (in this order) on a path from the root to a leaf of a tree $T$, and assume that $v_k$ is not a leaf. Moreover, $w_1$ and $w_2$ are children of $v_1$ and $v_k$ respectively such that $w_2$ is a leaf while $w_1$ is not. The leaf swap operation takes the subtree $B$ induced by $w_1$ and all its successors and swaps it with $w_2$ (equivalently, the edges $v_1w_1$ and $v_kw_2$ are removed and replaced by edges $v_1w_2$ and $v_kw_1$). This is illustrated in Figure \ref{fig:LSO}. 

\begin{figure}[htbp]
\begin{subfigure}[b]{0.4\textwidth}
\centering  
\begin{tikzpicture}[thick, level distance=10mm,scale=0.6]
\foreach \i in {2,4,6,8,10}{\draw (0,\i) node[shape=circle, fill, inner sep=2pt]{};}
\foreach \i in {2,6}{\draw (-2,\i) node[shape=circle, fill, inner sep=2pt]{};}
\draw (0,10)--(0,8) (0,6)--(0,4)(-2,6)--(0,6)(-2,2)--(0,2);
\draw[dashed] (0,8)--(0,6) (0,4)--(0,2);
\draw (0,10)--(2,10.5)--(2,9.5)--(0,10)(0,8)--(2,8.5)--(2,7.5)--(0,8)(0,6)--(2,6.5)--(2,5.5)--(0,6)(-2,6)--(-4,6.5)--(-4,5.5)--(-2,6)(0,4)--(2,4.5)--(2,3.5)--(0,4)(0,2)--(-1,0)--(1,0)--(0,2) ;
\draw (0,10.5)node{$r$}(-3.5,6)node{$B$}(-0.4,6.4)node{$v_1$}(1.5,6)node{$A_1$}(1.5,4)node{$A_2$}(-0.5,4)node{$v_2$}(0,0.9)node{$A_k$}(0.5,2)node{$v_k$}(-2,6.5)node{$w_1$}(-2,2.5)node{$w_2$};
\end{tikzpicture}
\caption{The rooted tree $T$.}
\end{subfigure}
\begin{subfigure}[b]{0.4\textwidth}
\centering
\begin{tikzpicture}[thick, level distance=10mm, scale=0.6]
\foreach \i in {2,4,6,8,10}{\draw (0,\i) node[shape=circle, fill, inner sep=2pt]{};}
\foreach \i in {2,6}{\draw (-2,\i) node[shape=circle, fill, inner sep=2pt]{};}
\draw (0,10)--(0,8) (0,6)--(0,4)(-2,6)--(0,6)(-2,2)--(0,2);
\draw[dashed] (0,8)--(0,6) (0,4)--(0,2);
\draw (0,10)--(2,10.5)--(2,9.5)--(0,10)(0,8)--(2,8.5)--(2,7.5)--(0,8)(0,6)--(2,6.5)--(2,5.5)--(0,6)(-2,2)--(-4,2.5)--(-4,1.5)--(-2,2)(0,4)--(2,4.5)--(2,3.5)--(0,4)(0,2)--(-1,0)--(1,0)--(0,2) ;
\draw (0,10.5)node{$r$}(-3.5,2)node{$B$}(-0.4,6.4)node{$v_1$}(1.5,6)node{$A_1$}(1.5,4)node{$A_2$}(-0.5,4)node{$v_2$}(0,0.9)node{$A_k$}(0.5,2)node{$v_k$}(-2,6.5)node{$w_2$}(-2,2.5)node{$w_1$};
\end{tikzpicture}
\caption{The rooted tree $T^\prime$.}
\end{subfigure}
\caption{$T^\prime$ is obtained from $T$ by the leaf swap operation.}\label{fig:LSO}
\end{figure}
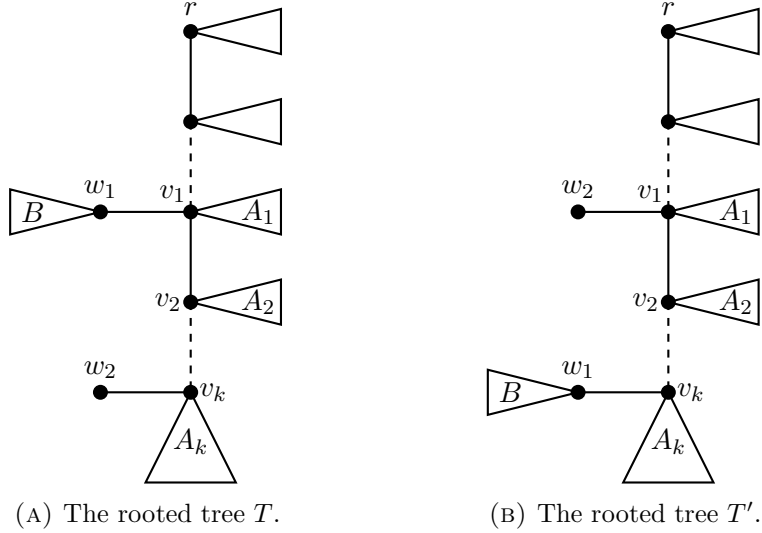

\begin{lem}\label{Lem:Swap}
Suppose that $T^\prime$ is obtained from a rooted tree $T$ by the leaf swap operation as described above and depicted in Figure \ref{fig:LSO}. Assume moreover that there is a Perron vector for $T$ with the property that the entries corresponding to leaf descendants of $v_k$ (all leaves for which the path to the root passes through $v_k$) are positive. Then $\rho_C (T^\prime) > \rho_C  (T)$.
\end{lem}

\begin{proof}
As indicated in Figure~\ref{fig:LSO}, we define subtrees $A_1,A_2,\ldots,A_k$ in a similar fashion as in the proof of Lemma~\ref{Lem:Branch_shifting}. Let $\f$ be a unit Perron vector for which the entries corresponding to $w_2$ and all leaves in $A_k$ are positive; such a vector exists by assumption. We write $f(u)$ for the entry of $\f$ corresponding to a leaf $u$ in $T$.
For a subtree $S$ of $T$, we set
\[|S|_{\f} =\sum_{x\in \CL (S)}f(x).\] 

To prove Lemma \ref{Lem:Swap}, we would once again like to show that the difference $R(C(T^\prime),\f) - R(C(T),\f)$ is positive. However, since this is not always the case, we will need to distinguish different cases. We have
\begin{align*}
R(C(T^\prime),\f) - R(C(T),\f) &= 2\sum_{\{x,y\}\subseteq \CL(T)}[\ell_{T^\prime}(x\vee y)-\ell_T(x\vee y)]f(x)f(y)\\
&\quad + \sum _{z\in \CL(T)} [\ell_{T^\prime}(z\vee z)-\ell_T(z\vee z)]f(z)^2.
\end{align*}
Now note that $\ell_{T^\prime}(x \vee y) =  \ell_{T}(x \vee y)$, unless $x$ or $y$ (or both) lie in $B \cup \{w_2\}$.  We consider cases where the difference $\ell_{T^\prime}(x\vee y)-\ell_T(x\vee y)$ is not equal to 0. This happens
\begin{itemize}
\item when $x\in B$ and $y \in A_i$, $i \in \{2,\ldots, k\}$; then we have
$\ell_{T^\prime}(x\vee y)-\ell_T(x\vee y) = i-1$.
\item when $x = w_2$ and $y \in A_i$, $i \in \{2,\ldots, k\}$; then we have
$\ell_{T^\prime}(x\vee y)-\ell_T(x\vee y) = 1-i$.
\item when $x, y \in B$; then we have $\ell_{T^\prime}(x\vee y)-\ell_T(x\vee y) = k-1$.
\item when $x = y = w_2$; then we have $\ell_{T^\prime}(x\vee y)-\ell_T(x\vee y) = 1-k$.
\end{itemize}
Combining all cases, we find that
\[R(C(T^\prime),\f) - R(C(T),\f) = 2\sum _{i=2}^k (i-1)L(A_i)_{\f}(|B|_{\f}-f(w_2)) + (k-1)(|B|_{\f}^2-f(w_2)^2).\]

Suppose first that $|B|_{\f} \geq f(w_2)$. Then it follows immediately that 
$$\rho_C(T^\prime) \geq R(C(T^\prime),\f) \geq R(C(T),\f) = \rho_C(T).$$
For equality to hold, $\f$ would also have to be an eigenvector for $C(T^\prime)$ corresponding to the eigenvalue $\rho_C(T^\prime) = \rho_C(T)$. But then it follows that
\begin{align*}
0 = \rho_C (T)f(w_2) - \rho_C(T^\prime)f(w_2) &= \sum_{v\in \CL(T)} \ell_T(v\vee w_2)f(v) - \sum_{v\in \CL(T^\prime)} \ell_{T^\prime}(v\vee w_2)f(v)\\
&= \sum_{v\in \CL(T)} [\ell_{T}(v\vee w_2) - \ell_{T^\prime}(v\vee w_2)]f(v)\\
&= (k-1)f(w_2) + \sum _{i=2}^k (i-1) |A_i|_{\f}\\
&\geq (k-1)f(w_2) > 0.
\end{align*}
This contradiction shows that equality cannot hold. Hence we have $\rho_C(T^\prime) > \rho_C(T)$ if $|B|_{\f} \geq f(w_2)$.

Now consider the second case that $0 < |B|_{\f}< f(w_2)$. Then we define a vector $\g$ whose entries are given by 
\[g(u) = \begin{cases}
|B|_{\f}  & \text{if } u =w_2,\\
\frac{f(u)f(w_2)}{|B|_{\f}} & \text{if } u \in B,\\
f(u) & \text{otherwise.}
\end{cases}
\]
This definition implies that $|B|_{\g} = f(w_2)$ and $|B|_{\f} = g(w_2)$. Moreover, we have $\ell_{T^\prime}(x \vee y) = \ell_{T}(x \vee y)$ if $x,y \notin \CL(B) \cup \{w_2\}$, $\ell_{T^\prime}(x \vee y) = \ell_{T}(w_2 \vee y)$ and $\ell_{T^\prime}(w_2 \vee y) = \ell_{T}(x \vee y)$ if $x \in \CL(B)$ and $y \notin \CL(B) \cup \{w_2\}$, and finally $\ell_{T^\prime}(x \vee w_2) = \ell_{T}(x \vee w_2) = \ell_{T}(v_1)$ if $x \in \CL(B)$. This means that most terms in the difference $\g^t C(T^\prime) \g - \f^t C(T) \f$ cancel. We are only left with
\begin{align*}
\g^t C(T^\prime) \g - \f^t C(T) \f &= \sum_{x \in \CL(B)} \sum_{y \in \CL(B)} [g(x)g(y) \ell_{T^\prime}(x \vee y) - f(x)f(y)\ell_{T}(x \vee y)] \\
&\quad + g(w_2)^2 \ell_{T^\prime}(w_2) - f(w_2)^2 \ell_T(w_2) \\
&= \sum_{x \in \CL(B)} \sum_{y \in \CL(B)} [(g(x)g(y) -f(x)f(y)) \ell_{T}(x \vee y) + (k-1)g(x)g(y)] \\
&\quad + (g(w_2)^2-f(w_2)^2)\ell_T(w_2) - g(w_2)^2 (k-1) \\
&= \Big( \frac{f(w_2)^2}{|B|_{\f}^2} - 1 \Big) \sum_{x \in \CL(B)} \sum_{y \in \CL(B)} f(x)f(y) \ell_{T}(x \vee y) + (k-1) |B|_{\g}^2 \\
&\quad + (g(w_2)^2-f(w_2)^2)\ell_T(w_2) - g(w_2)^2 (k-1).
\end{align*}
Since $\ell_{T}(x \vee y) \geq \ell_{T}(v_1) + 1$ for all $x,y \in \CL(B)$, it follows that
\begin{align*}
\g^t C(T^\prime) \g - \f^t C(T) \f &\geq \Big( \frac{f(w_2)^2}{g(w_2)^2} - 1 \Big) \sum_{x \in \CL(B)} \sum_{y \in \CL(B)} f(x)f(y) (\ell_{T}(v_1) + 1) + (k-1) f(w_2)^2 \\
&\quad + (g(w_2)^2-f(w_2)^2)\ell_T(w_2) - g(w_2)^2 (k-1) \\
&= \Big( \frac{f(w_2)^2}{g(w_2)^2} - 1 \Big) |B|_{\f}^2(\ell_{T}(v_1) + 1) + (g(w_2)^2 - f(w_2)^2)(\ell_T(w_2) - k+1)\\
&= \Big( \frac{f(w_2)^2}{g(w_2)^2} - 1 \Big) g(w_2)^2(\ell_{T}(v_1) + 1) + (g(w_2)^2 - f(w_2)^2)(\ell_T(v_1) +1)\\
&= 0.
\end{align*}
Moreover, we have (since $\f$ was assumed to be a unit vector)
$$1 = \|\f\|^2 = f(w_2)^2 + \sum _{u\in \CL(B)} f(u)^2 + \sum _{u \notin \CL (B) \cup \{w_2\}} f(u)^2$$
and
$$\|\g\|^2 = |B|_{\f}^2 + \frac{f(w_2)^2}{|B|_{\f}^2}\Big(\sum _{u\in \CL (B)} f(u)^2 \Big) + \sum _{u \notin \CL (B) \cup \{w_2\}} f(u)^2,$$
thus
\begin{align*}
\|\f\|^2-\|\g\|^2 &= f(w_2)^2\Big(1-\frac{\sum _{u\in \CL (B)} f(u)^2}{|B|_{\f}^2}\Big) + \sum _{u\in \CL (B)} f(u)^2 - |B|_{\f}^2\\
&= f(w_2)^2\Big(\frac{|B|_{\f}^2- \sum _{u\in \CL (B)} f(u)^2}{|B|_{\f}^2}\Big) - \Big( |B|_{\f}^2-\sum _{u\in \CL (B)} f(u)^2\Big)\\
&=\Big( |B|_{\f}^2-\sum _{u\in \CL (B)} f(u)^2\Big)\Big(\frac{f(w_2)^2}{|B|_{\f}^2}-1\Big).
\end{align*}
The second factor is strictly positive since  $|B|_{\f} < f(w_2)$ by assumption. The first factor is non-negative, since we can write it as
$$|B|_{\f}^2 - \sum_{u \in \CL(B)} f(u)^2 = \sum_{u \in \CL(B)}\sum_{v \in \CL(B) \setminus \{u\}} f(u)f(v).$$
Moreover, since $|B|_{\f} > 0$, we must have $f(u) > 0$ for all $u \in \CL(B)$ (if one of them is positive, all of them are, since the leaves of $B$ belong to the same root branch). So if $B$ contains at least two leaves, then the first factor is also strictly positive. 

Thus we conclude that $\|\g\| \leq \|\f\| = 1$, with strict inequality if $B$ contains at least two leaves. If this is the case, we can combine it with the inequality
\begin{equation}\label{eq:rayleigh_intermediate}
\g^t C(T^\prime) \g - \f^t C(T) \f \geq 0
\end{equation}
 that was proven earlier to obtain
\begin{equation}\label{eq:rayleigh_final}
\rho_C(T^\prime) \geq R(C(T^\prime),\g) > R(C(T),\f) = \rho_C(T).
\end{equation}
If $B$ only contains one leaf, then the assumption that $B$ is not just a single vertex allows us to replace the inequality $\ell_{T}(x \vee y) \geq \ell_{T}(v_1) + 1$ that was used earlier by the stronger version $\ell_{T}(x \vee y) \geq \ell_{T}(v_1) + 2$, giving us strict inequality in~\eqref{eq:rayleigh_intermediate}. Once again, we have~\eqref{eq:rayleigh_final}, which completes the proof in this case.

Finally, in the third case that $|B|_{\f} = 0$, we set
\[g(u) = \begin{cases}
0  & \text{if } u =w_2,\\
\frac{f(w_2)}{L(B)} & \text{if } u \in B,\\
f(u) & \text{otherwise.}
\end{cases}
\]
Then we can proceed in exactly the same way as in the second case.
\end{proof}

\begin{rem}
In each of the three preceding lemmas, the assumption on the Perron vector is essential to ensure strict inequality. Otherwise, we only get $\rho_C(T^\prime) \geq \rho_C(T)$ from the three operations.
\end{rem}


A greedy caterpillar (see Figure~\ref{fig:greedy}) is a rooted caterpillar with the property that the outdegrees of its internal vertices increase along the backbone. Given an outdegree sequence $S$, we can construct a greedy caterpillar $G(S)$ by the following steps:
\begin{itemize}
\item Construct the backbone, which is a path whose length is the number of non-zero entries in the outdegree sequence $S$.
\item Assign the lowest non-zero entry in $S$ (say, $s$) to the root (one end of the backbone) by attaching $s-1$ leaves to it.
\item Assign, in ascending order, non-zero entries in $S$ to the vertices on the backbone with respect to their distance from the root by attaching a suitable number of leaves.
\end{itemize}

\begin{figure}[htbp]\centering  
\begin{tikzpicture}[thick, level distance=12mm, 
  every node/.style = {shape=circle, fill, inner sep=2pt}]
\tikzstyle{level 1}=[sibling distance=10mm]
\tikzstyle{level 2}=[sibling distance=10mm]
\tikzstyle{level 3}=[sibling distance=10mm]
\tikzstyle{level 4}=[sibling distance=10mm]
\tikzstyle{level 5}=[sibling distance=10mm]
\node {}
child {node {}
child {node {}
child {node {}}
child {node {}
child {node {}}
child {node {}}
child {node {}
child {node {}}
child {node {}}
child {node {}}
child {node {}}
child {node {}}
}
child {node {}}
child {node {}}
}
child {node {}}
}}
;

\end{tikzpicture}
\caption{A greedy caterpillar with outdegree sequence $(5,5,3,1,1,0,\ldots,0)$.}\label{fig:greedy}
\end{figure}
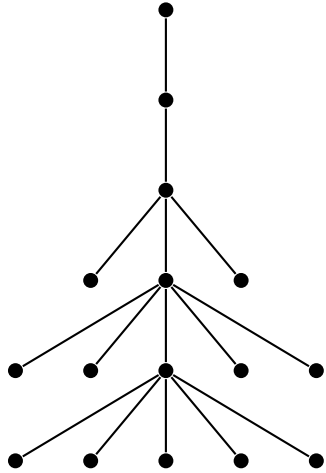

\begin{thm}\label{thm:greedy}
Among all trees with outdegree sequence $S$, the greedy caterpillar $G(S)$ has the maximum ancestral spectral radius.
\end{thm}

\begin{proof}
To prove this theorem we show that a greedy caterpillar $T^\prime$ can be obtained from any rooted tree $T$ with the same outdegree sequence by the branch shift and the leaf swap operation. Hence by Lemma \ref{Lem:Branch_shifting} and Lemma \ref{Lem:Swap}, we obtain the statement of the theorem.

If $T$ is already a caterpillar, but not a greedy caterpillar, then we can easily obtain the greedy caterpillar $G(S)$ by shifting leaves along its backbone (repeatedly applying the branch shift operation) without changing the outdegree sequence. Therefore by Lemma \ref{Lem:Branch_shifting}, we get that $\rho_C (G(S)) > \rho_C (T)$. 

Otherwise, we transform $T$ into a caterpillar. Consider the vertex closest to the root (possibly the root itself) that has more than one non-leaf child, and call it $v_1$. Moreover, let $\f$ be a Perron vector associated with $T$. At least one of the branches rooted at the children of $v_1$ must contain leaves for which the corresponding entries of $\f$ are positive: if not, then the positive entries of $\f$ would have to lie in a root branch to which $v_1$ does not belong, which would have to be a leaf. In this case, we would have $\rho_C(T) = 1$, which is impossible.

In the aforementioned branch of $v_1$, we can find a vertex $v_k$ with a leaf child $w_2$, and there is also a branch of $v_1$ (rooted at a child $w_1$ of $v_1$) that is not just a single vertex. Thus we are in the situation where the leaf swap operation applies, so we can construct a new tree $T^\prime$ with the same outdegree sequence such that $\rho_C (T^\prime) > \rho_C (T)$. This procedure can be repeated until we obtain a caterpillar.
\end{proof}

There are similar examples where ``greedy'' structures maximise the spectral radius; for instance, this is the case for the spectral radius and Laplacian spectral radius of trees \cite{Biyikoglu2008graphs,Zhang2008Laplacian}.

An immediate consequence of Theorem~\ref{thm:greedy} is Theorem~\ref{Thm:rootedBroom} below, which deals with trees for which the number of leaves and the number of vertices are given. A rooted broom $B_{m,n}$ is a rooted tree obtained by attaching $n$ leaves to one end of a path of length $m$, the other end being the root. Thus a rooted broom $B_{m,n}$ has $n$ leaves and $m+n+1$ vertices.

\begin{figure}[htbp]
\centering  
\begin{tikzpicture}[thick, level distance=10mm,scale=0.6]
\foreach \i in {2,4,6,8}{\draw (0,\i) node[shape=circle, fill, inner sep=2pt]{};\draw (0,2)--(0,\i);}
\foreach \i in {-1,1}{\draw (\i,2) node[shape=circle, fill, inner sep=2pt]{};}
\draw (-1,2)--(0,4) (1,2)--(0,4);
\end{tikzpicture}
\caption{The rooted broom $B(2,3)$.}\label{fig:Rooted_broom}
\end{figure}
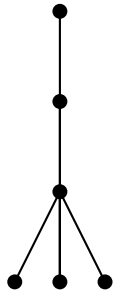

\begin{thm}\label{Thm:rootedBroom}
The rooted broom $B_{N-n-1,n}$ maximises the ancestral spectral radius among all rooted trees with $N$ vertices and $n$ leaves. We have
$$\rho_C(B_{N-n-1,n}) = n(N - n - 1) + 1.$$
\end{thm}

\begin{proof}

By the previous theorem, it is clear that the tree with greatest ancestral spectral radius among all such rooted trees has to be a greedy caterpillar. Applying the branch shift operation repeatedly to the leaves of such a greedy caterpillar to transfer all leaves to the lowest internal vertex, we obtain a rooted broom at the end, and the ancestral spectral radius increases with each step by Lemma \ref{Lem:Branch_shifting}.  

To complete the proof, we only need to determine the value of $\rho_C(B_{N-n-1,n})$. Here, we note that the entries of $C(B_{N-n-1,n})$ are all equal to $N-n -1$, except for those on the diagonal, which are equal to $N-n$. We see that $1$ is an eigenvalue of multiplicity $n-1$, the remaining eigenvalue being $n(N-n) - (n-1) = n(N-n-1) + 1$.
\end{proof}

Next, we consider another type of restriction on the degrees. We first observe that $\rho_C(T)$ is unbounded even if the number of leaves $L(T)$ is fixed: for instance, one can consider the rooted brooms from the previous theorem. This changes, however, if we forbid vertices of outdegree $1$. A tree with this property is called homeomorphically irreducible, series-reduced or topological. It turns out that the binary caterpillar tree is extremal in this case. The binary caterpillar $C_n$ (see Figure~\ref{fig:binary_cater}) is the rooted tree in which all $n-1$ internal vertices form a path with the root at one of its ends, and each of them has precisely two children. Note that there are precisely $n$ leaves. We have the following theorem.

\begin{figure}[htbp]\centering  
\begin{tikzpicture}[thick, level distance=12mm, 
  every node/.style = {shape=circle, fill, inner sep=2pt}]
\tikzstyle{level 1}=[sibling distance=10mm]
\tikzstyle{level 2}=[sibling distance=10mm]
\tikzstyle{level 3}=[sibling distance=10mm]
\tikzstyle{level 4}=[sibling distance=10mm]
\node {}
child {node {}}
child {node {}
child {node {}}
child {node {}
child {node {}}
child {node {}
child {node {}}
child {node {}
}}}}
;

\end{tikzpicture}
\caption{The binary caterpillar $C_5$.}\label{fig:binary_cater}
\end{figure}
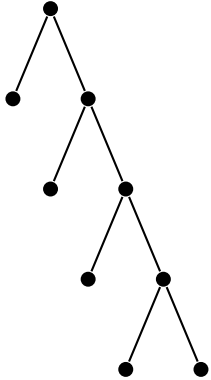

\begin{thm}\label{thm:caterpillar}
For every rooted tree $T$ with $n$ leaves and no vertices whose outdegree is $1$, we have
$$\rho_C(T) \leq \rho_C(C_n).$$
\end{thm}

\begin{proof}
Again, Theorem~\ref{thm:greedy} immediately shows that the maximum has to be attained by a greedy caterpillar. If this caterpillar is not the binary caterpillar $C_n$, then the internal vertex whose distance from the root is greatest must have at least three children. But then we can apply the star shift operation
to obtain a new tree that still does not contain any vertices of outdegree $1$ whose ancestral spectral radius is greater by Lemma~\ref{Lem:Star_shift}.

This contradiction shows that $C_n$ must indeed attain the maximum, which is exactly the statement of the theorem.
\end{proof}

Theorem~\ref{thm:caterpillar} raises the question for the value of $\rho_C(C_n)$. While there is no exact formula, we will be able to provide an implicit equation and an asymptotic formula in the following. To this end, we first determine a recursion for the characteristic polynomial of $C(C_n)$: set
\begin{equation}\label{eq:char_poly_Cn}
P_n(x) = \det(x I - C(C_n)).
\end{equation}

\begin{prop}\label{prop:caterpillar_rec}
Let $P_n(x)$ be the characteristic polynomial of the ancestral matrix of the binary caterpillar $C_n$, as defined in~\eqref{eq:char_poly_Cn}.  The following recursion holds:
$$P_n(x) = (2x-3)P_{n-1}(x) - (x-1)^2 P_{n-2}(x),$$
with initial values $P_1(x) = x$ and $P_2(x) = (x-1)^2$. 
\end{prop}

\begin{proof}
The initial values are easily determined from the matrices
$$C(C_1) = \begin{bmatrix} 0 \end{bmatrix} \qquad \text{and} \qquad C(C_2) = \begin{bmatrix} 1 & 0 \\ 0 & 1 \end{bmatrix},$$
so we focus on the recursion. The ancestral matrix of the caterpillar $C_n$ has the form
$$C(C_n) = \begin{bmatrix}
1 & 0 & 0 & 0 & \cdots & 0 & 0 \\
0 & 2 & 1 & 1 & \cdots & 1 & 1 \\
0 & 1 & 3 & 2 & \cdots & 2 & 2 \\
0 & 1 & 2 & 4 & \cdots & 3 & 3 \\
\vdots & \vdots & \vdots & \vdots & \ddots & \vdots & \vdots \\
0 & 1 & 2 & 3 & \cdots & n-1 & n-2 \\
0 & 1 & 2 & 3 & \cdots & n-2 & n-1
\end{bmatrix}.$$
We will also need the following auxiliary matrix, which only differs in the last entry:
$$H_n = \begin{bmatrix}
1 & 0 & 0 & 0 & \cdots & 0 & 0 \\
0 & 2 & 1 & 1 & \cdots & 1 & 1 \\
0 & 1 & 3 & 2 & \cdots & 2 & 2 \\
0 & 1 & 2 & 4 & \cdots & 3 & 3 \\
\vdots & \vdots & \vdots & \vdots & \ddots & \vdots & \vdots \\
0 & 1 & 2 & 3 & \cdots & n-1 & n-2 \\
0 & 1 & 2 & 3 & \cdots & n-2 & n
\end{bmatrix}.$$
Note that the $(n-1) \times (n-1)$ submatrix obtained by removing the last row and column is the same for $C(C_n)$ and $H_n$, namely $H_{n-1}$. Using the linearity of the determinant with respect to the last row, we get
\begin{align*}
P_n(x) &= \det(x I - C(C_n)) =
\begin{vmatrix}
x-1 & 0 & \cdots & 0 & 0 \\
0 & x-2 & \cdots & -1 & -1 \\
\vdots & \vdots & \ddots & \vdots & \vdots \\
0 & -1 & \cdots & x-n+1 & 2-n \\
0 & -1 & \cdots & 2-n & x-n+1
\end{vmatrix} \\
&= \begin{vmatrix}
x-1 & 0 & \cdots & 0 & 0 \\
0 & x-2 & \cdots & -1 & -1 \\
\vdots & \vdots & \ddots & \vdots & \vdots \\
0 & -1 & \cdots & x-n+1 & 2-n \\
0 & -1 & \cdots & 2-n & x-n
\end{vmatrix}
+
\begin{vmatrix}
x-1 & 0 & \cdots & 0 & 0 \\
0 & x-2 & \cdots & -1 & -1 \\
\vdots & \vdots & \ddots & \vdots & \vdots \\
0 & -1 & \cdots & x-n+1 & 2-n \\
0 & 0 & \cdots & 0 & 1
\end{vmatrix},
\end{align*}
so
\begin{equation}\label{eq:pathdet1}
P_n(x) = \det(x I - C(C_n)) = \det(x I - H_n) + \det(x I - H_{n-1}).
\end{equation}
On the other hand, subtracting the second-to-last row from the last, then the second-to-last column from the last, we find that 
\begin{align*}
P_n(x) &= \det(x I - C(C_n)) \\
&= 
\begin{vmatrix}
x-1 & 0 & 0 & 0 & \cdots & 0 & 0 \\
0 & x-2 & -1 & -1 & \cdots & -1 & -1 \\
0 & -1 & x-3 & -2 & \cdots & -2 & -2 \\
0 & -1 & -2 & x-4 & \cdots & -3 & -3 \\
\vdots & \vdots & \vdots & \vdots & \ddots & \vdots & \vdots \\
0 & -1 & -2 & -3 & \cdots & x-n+1 & 2-n \\
0 & -1 & -2 & -3 & \cdots & 2-n & x-n+1
\end{vmatrix} \\
&= \begin{vmatrix}
x-1 & 0 & 0 & 0 & \cdots & 0 & 0 \\
0 & x-2 & -1 & -1 & \cdots & -1 & -1 \\
0 & -1 & x-3 & -2 & \cdots & -2 & -2 \\
0 & -1 & -2 & x-4 & \cdots & -3 & -3 \\
\vdots & \vdots & \vdots & \vdots & \ddots & \vdots & \vdots \\
0 & -1 & -2 & -3 & \cdots & x-n+1 & 2-n \\
0 & 0 & 0 & 0 & \cdots & 1-x & x-1
\end{vmatrix} \\
&= \begin{vmatrix}
x-1 & 0 & 0 & 0 & \cdots & 0 & 0 \\
0 & x-2 & -1 & -1 & \cdots & -1 & 0 \\
0 & -1 & x-3 & -2 & \cdots & -2 & 0 \\
0 & -1 & -2 & x-4 & \cdots & -3 & 0 \\
\vdots & \vdots & \vdots & \vdots & \ddots & \vdots & \vdots \\
0 & -1 & -2 & -3 & \cdots & x-n+1 & 1-x \\
0 & 0 & 0 & 0 & \cdots & 1-x & 2x-2
\end{vmatrix}.
\end{align*}
Now use row expansion with respect to the last row, followed by column expansion with respect to the last column. This yields
$$P_n(x) = (2x-2) \det(xI - H_{n-1}) - (1-x)^2 \det(xI - H_{n-2}).$$
Combining this with~\eqref{eq:pathdet1}, we find
$$\det(x I - H_n) = (2x-3) \det(xI - H_{n-1}) - (1-x)^2 \det(xI - H_{n-2}).$$
Invoking~\eqref{eq:pathdet1} once again, we end up with
\begin{align*}
P_n(x) &= \det(x I - H_n) + \det(x I - H_{n-1}) \\
&= (2x-3) \big(\det(xI - H_{n-1}) + \det(xI - H_{n-2}) \big) \\
&\quad - (1-x)^2 \big( \det(xI - H_{n-2}) + \det(xI - H_{n-3}) \big) \\
&= (2x-3)P_{n-1}(x) - (x-1)^2 P_{n-2}(x).
\end{align*}
This completes the proof.
\end{proof}

The polynomial $P_n(x)$ can be expressed in terms of Chebyshev polynomials. Chebyshev polynomials also occur, for example, in the characteristic polynomials (with respect to the adjacency matrix) of the path and the cycle, see e.g. \cite[Section 3.1]{Aigner}. Recall that the Chebyshev polynomials $T_n(x)$ and $U_n(x)$ are given by the recursions
$$T_0(x) = 1,\ T_1(x) = x,\ T_n(x) = 2x T_{n-1}(x) - T_{n-2}(x)$$
and
$$U_0(x) = 1,\ U_1(x) = 2x,\ U_n(x) = 2x U_{n-1}(x) - U_{n-2}(x).$$
If we substitute $P_n(x) = (x-1)^n Q_n(x)$, then the recursion of Proposition~\ref{prop:caterpillar_rec} becomes
$$Q_n(x) = \frac{2x-3}{x-1} Q_{n-1}(x) - Q_{n-2}(x),$$
with the initial values $Q_1(x) = \frac{x}{x-1}$ and $Q_2(x) = 1$. We observe that $Q_n(x)$ satisfies the same linear recursion as $T_n(\frac{2x-3}{2x-2})$ and $U_n(\frac{2x-3}{2x-2})$ or any linear combination of these two. It is easy to verify that
$$\frac{2x}{2x-3} T_{n-2} \Big( \frac{2x-3}{2x-2} \Big) - \frac{3}{2x-3} U_{n-2} \Big( \frac{2x-3}{2x-2} \Big)$$
has the same values for $n=2$ and $n=3$ as $Q_n$, and since they satisfy the same second-order linear recursion, we must have
$$Q_n(x) = \frac{2x}{2x-3} T_{n-2} \Big( \frac{2x-3}{2x-2} \Big) - \frac{3}{2x-3} U_{n-2} \Big( \frac{2x-3}{2x-2} \Big)$$
for all $n \geq 2$, thus
\begin{equation}\label{eq:Pn_ito_chebyshev}
P_n(x) = (x-1)^n \Big( \frac{2x}{2x-3} T_{n-2} \Big( \frac{2x-3}{2x-2} \Big) - \frac{3}{2x-3} U_{n-2} \Big( \frac{2x-3}{2x-2} \Big) \Big).
\end{equation}
The Chebyshev polynomials are well known to be connected to trigonometric functions by the identities
$$T_n(\cos t) = \cos(n t)\qquad \text{and} \qquad U_n(\cos t) = \frac{\sin((n+1)t)}{\sin t}.$$
This motivates the substitution $\frac{2x-3}{2x-2} = \cos t$ (equivalently, $x = 1 + \frac{1}{4\sin^2 (t/2)}$) in~\eqref{eq:Pn_ito_chebyshev}, which gives us
$$P_n(x) = (2 \sin (t/2))^{-2n} \Big( \Big( \frac{3}{\cos t} - 2 \Big) \cos((n-2)t) - \Big( \frac{3}{\cos t} - 3 \Big) \frac{\sin((n-1)t)}{\sin t} \Big).$$
Now use the addition theorem for the sine function to rewrite $\sin((n-1)t)$ as $\sin((n-2)t) \cos t + \cos((n-2)t) \sin t$, which results in the following simplified expression:
$$P_n(x) = \frac{1}{(2 \sin (t/2))^{2n}\sin t} \Big( \sin t \cos((n-2)t) - 3(1-\cos t) \sin((n-2)t) \Big).$$
Thus the characteristic equation $P_n(x) = 0$ reduces to
$$\sin t \cos((n-2)t) = 3(1-\cos t) \sin((n-2)t)$$
or
\begin{equation}\label{eq:char_eq_trig}
\cot ((n-2)t) = \frac{3(1-\cos t)}{\sin t} = 3 \tan(t/2).
\end{equation}
The following asymptotic formula is now a fairly straightforward consequence:

\begin{thm}
The spectral radius of the ancestral matrix of the caterpillar $C_n$ satisfies the asymptotic formula
$$\rho_C(C_n) = \frac{4n^2}{\pi^2} - \frac{4n}{\pi^2} + O(1)$$
as  $n \to \infty$.
\end{thm}

\begin{proof}
Recall that $x = 1 + \frac{1}{4\sin^2 t/2}$ in~\eqref{eq:char_eq_trig}. The greatest eigenvalue thus corresponds to the smallest positive value of $t$ (which we will denote by $t_0$) that satisfies
\begin{equation}\label{eq:implicit_t}
\cot ((n-2)t) = 3 \tan(t/2).
\end{equation}
For large $n$, we know that the right side is positive and increasing for $t \in (0,\frac{\pi}{2(n-2)})$, while the left side is decreasing and covers the entire range from $0$ to $\infty$. Thus there must be a (unique) solution in that interval by the intermediate value theorem. Since we are looking for the smallest positive $t$ that satisfies~\eqref{eq:implicit_t}, we can conclude that $t_0 < \frac{\pi}{2(n-2)}$, thus $t_0 \to 0$ as $n \to \infty$. So $\cot ((n-2)t_0)$ must be close to $0$, and since $\frac{\pi}{2}$ is the smallest positive zero of the cotangent, we infer that $t_0 \sim \frac{\pi}{2n}$. Thus
$$\cot((n-2)t_0) = 3\tan(t_0/2) \sim \frac{3\pi}{4n},$$
and the Taylor approximation of the cotangent yields a second-order approximation for $t_0$:
$$\cot((n-2)t_0) \sim \frac{\pi}{2} - (n-2)t_0,$$
which results in
$$t_0 - \frac{\pi}{2n} \sim -\frac{1}{n} \Big( \cot((n-2)t_0) - 2t_0 \Big) = -\frac{1}{n} \Big( 3\tan(t_0/2) - 2t_0 \Big) \sim \frac{t_0}{2n} \sim \frac{\pi}{4n^2}.$$
Continuing in this way, one could even determine further terms of an asymptotic expansion. Since our initial substitution was $x = 1 + \frac{1}{4\sin^2 t/2}$, we have to plug the formula
$$t_0 = \frac{\pi}{2n} + \frac{\pi}{4n^2} + O(n^{-3})$$
in for $t$, which gives us
$$\rho(C_n) = 1 + \frac{1}{4\sin^2 t_0/2} = \frac{4n^2}{\pi^2} - \frac{4n}{\pi^2} + O(1).$$
\end{proof}

\section{Determinants and the characteristic polynomial}

In this section, we take a closer look at the characteristic polynomial
$$\Gamma_T(x) = \det(x I - C(T)) = \sum_{k=0}^n (-1)^{k} \gamma_k(T) x^{n-k}.$$
Specifically, we will determine a combinatorial interpretation for the coefficients $\gamma_k(T)$ of this polynomial, similar to the classical combinatorial formulas due to Sachs \cite{Sachs1964beziehungen} and Kelmans \cite{Kelmans1967properties} for the coefficients of the characteristic polynomials of the adjacency matrix and the Laplacian matrix, respectively (see e.g. \cite[Sections 1.4 and 1.5]{Cvetkovic1995spectra}). Moreover, using a recursive approach similar to the proof of Proposition~\ref{prop:caterpillar_rec}, we find that specific values of the characteristic polynomial are independent of the precise structure of the tree when $d$-ary trees are considered, see Theorem~\ref{thm:det_dary}.

A very basic observation can be made about the coefficient $\gamma_1(T)$, which is equal to the trace of $C(T)$ and thus also the sum of the eigenvalues.

\begin{prop}
Let $T$ be a rooted tree with root $r$ and $n$ leaves, and let $\alpha_1,\alpha_2,\ldots,\alpha_n$ be the eigenvalues of its ancestral matrix $C(T)$. We have
$$\gamma_1(T) = \sum_{k=1}^n \alpha_k  = \operatorname{tr}(C(T)) = D_T(r),$$
where $D_T(r)$ denotes the sum of the distances of all leaves to the root.
\end{prop}

\begin{proof}
The identity of the first three expressions is basic linear algebra. To see that $\operatorname{tr}(C(T))$ equals $D_T(r)$, simply note that the diagonal entry in $C(T)$ corresponding to a leaf equals the level (distance to the root) of that leaf.
\end{proof}

The general interpretation of the coefficients $\gamma_k(T)$ is somewhat more involved. We need to start with a few definitions. An upward path from a leaf is a path (potentially trivial, i.e., only consisting of the leaf itself, without any edges) starting at a leaf and only moving towards the root. It is easy to see that a vertex at level $\ell$ has $\ell+1$ upward paths emanating from it, including the trivial path. This will be important later. 
We will be specifically interested in collections of upward paths in a rooted tree, one starting from each of the leaves, that are edge-disjoint (not necessarily vertex-disjoint). See Figure~\ref{fig:nc_collection} for an example of an edge-disjoint collection in the tree of Figure~\ref{fig:first_example}.
These edge-disjoint collections are counted by the coefficients of the characteristic polynomial of $C(T)$. The following result and its proof are reminiscent of the well-known Lindstr\"om-Gessel-Viennot Lemma (\cite{Gessel1985Binomial}, see also for example \cite[Section 5.4]{Aigner}).

\begin{figure}[htbp]\centering  
\begin{tikzpicture}[dashed, level distance=12mm, 
  every node/.style = {shape=circle, fill, inner sep=2pt}]
\tikzstyle{level 1}=[sibling distance=20mm]
\tikzstyle{level 2}=[sibling distance=10mm]
\tikzstyle{level 2}=[sibling distance=10mm]
\node {}
child {node {}
child {node {}} 
child {node {}}}
child {node {} 
child {node {}} 
child {node {}}
child {node {}
child {node{}}
child {node{}}
}};

\node [draw=none,fill=none] at (-1.5,-2.8) {$v_1$};
\node [draw=none,fill=none] at (-0.5,-2.8) {$v_2$};
\node [draw=none,fill=none] at (0,-2.8) {$v_3$};
\node [draw=none,fill=none] at (1,-2.8) {$v_4$};
\node [draw=none,fill=none] at (1.5,-4) {$v_5$};
\node [draw=none,fill=none] at (2.5,-4) {$v_6$};

\draw [solid, thick,->] (-0.5,-2.4) -- (-0.95,-1.32);
\draw [solid, thick,->] (0,-2.4) -- (1,-1.2) -- (0.1,-0.12);
\draw [solid, thick,->] (1.5,-3.6) -- (1.95,-2.52);
\draw [solid, thick,->] (2.5,-3.6) -- (2,-2.4) -- (1.1,-1.32);

\end{tikzpicture}
\caption{An edge-disjoint collection (the paths emanating from $v_1$ and $v_4$ are trivial).}\label{fig:nc_collection}
\end{figure}
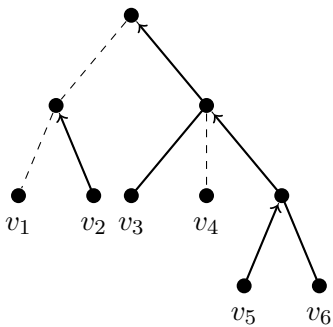

\begin{thm}\label{thm:edge-disjoint}
Let
$$\Gamma_T(x) = \det(x I - C(T)) = \sum_{k=0}^n (-1)^{k} \gamma_k(T) x^{n-k}$$
be the characteristic polynomial of the ancestral matrix $C(T)$ of a rooted tree $T$. The coefficient $\gamma_k(T)$ is the number of edge-disjoint collections of upward paths where exactly $k$ of the paths are non-trivial. Consequently,
$$\det(I + C(T)) = (-1)^n \Gamma_T(-1) = \sum_{k=0}^n \gamma_k(T)$$
is the total number of edge-disjoint collections of upward paths.
\end{thm}

\begin{proof}
It is slightly more convenient for the proof to replace $x$ by $-x$ and consider
$$\det(x I + C(T)) = (-1)^n \Gamma_T(-x) = \sum_{k=0}^n \gamma_k(T) x^{n-k}.$$
Let $m_{ij}$ be the entry in the $i$-th row, $j$-th column of $x I + C(T)$. By definition of $C(T)$, we have
$$m_{ij} = \begin{cases} x + \ell(v_i) & i = j, \\ \ell(v_i \vee v_j) & i \neq j. \end{cases}$$
We apply the Leibniz formula for the determinant to obtain
\begin{equation}\label{eq:leibniz}
\det(x I + C(T)) = \sum_{\sigma \in S_n} \sgn \sigma \prod_{i=1}^n m_{i,\sigma(i)}.
\end{equation}
We say that a collection $\Pc$ of upward paths is compatible with a permutation $\sigma$ if the following holds for all $i$:
\begin{itemize}
\item the upward path $P_i$ starting at $v_i$ is trivial, and $\sigma(i) = i$, or
\item both ends of the last edge of $P_i$ are ancestors of $v_{\sigma(i)}$ (possibly, one of them is $v_i$ itself if $\sigma(i) = i$).
\end{itemize}
Note that $\ell(v_i \vee v_{\sigma(i)})$ is the number of possibilities for $P_i$ satisfying this property, except when $i = \sigma(i)$. In the latter case, the number of possibilities is $\ell(v_i) + 1$, since the trivial path is included as well. So writing $e(\Pc)$ for the number of trivial paths occurring in a collection $\Pc$, we get
$$\prod_{i=1}^n m_{i,\sigma(i)} = \sum_{\substack{\Pc \\ \Pc, \sigma \text{ compatible}}} x^{e(\Pc)}.$$
We plug this into~\eqref{eq:leibniz} and interchange the order of summation:
\begin{equation}\label{eq:det_path_collections}
\det(x I + C(T)) = \sum_{\Pc} x^{e(\Pc)} \sum_{\substack{\sigma \\ \Pc, \sigma \text{ compatible}}} \sgn \sigma.
\end{equation}
Suppose first that $\Pc$ is not an edge-disjoint collection (thus an ``intersecting'' collection). We construct another collection $\Pc^*$ in the following way: consider the (lexicographically) smallest pair of indices $i,j$ such that the paths $P_i$ and $P_j$ emanating respectively from $v_i$ and $v_j$ have a common edge. Now $\Pc^*$ is obtained by interchanging the parts of $P_i$ and $P_j$ starting from the lowest common edge (going up). It is clear that this defines an involution on the set of intersecting collections of upward paths. Importantly, if $\Pc$ is compatible with $\sigma$, then $\Pc^*$ is compatible with a permutation $\sigma^*$ that differs from $\sigma$ only by a transposition of $i$ and $j$. Since $\sigma$ and $\sigma^*$ have opposite signs, it follows that
$$\sum_{\substack{\sigma \\ \Pc, \sigma \text{ compatible}}} \sgn \sigma = - \sum_{\substack{\sigma \\ \Pc^*, \sigma \text{ compatible}}} \sgn \sigma,$$
which means that all intersecting collections $\Pc$ cancel pairwise in~\eqref{eq:det_path_collections} (if $\Pc = \Pc^*$, then the sum over $\sigma$ is $0$). Thus we are left to consider edge-disjoint collections. 

Now we claim that the only permutation that is compatible with an edge-disjoint collection is the identity, from which the desired formula follows immediately. We prove this claim by induction on the number of leaves. If there is only a single leaf, then the identity is the only permutation, so the claim is trivial. Otherwise, let $\Pc$ be an edge-disjoint collection of upward paths, and consider an internal vertex $w$ with more than one child whose level is maximal among all such vertices. Let $v_{i_1},v_{i_2},\ldots,v_{i_r}$ be the leaves of which $w$ is an ancestor. By the choice of $w$, $r \geq 2$, and the paths from $w$ to these leaves are pairwise edge-disjoint. Since $\Pc$ is an edge-disjoint collection, there must be at least one leaf $v_{i_s}$ ($s \in \{1,2,\ldots,r\}$) such that the upward path starting from $v_{i_s}$ does not go beyond $w$. Thus for $\sigma$ to be compatible with $\Pc$, we need to have $\sigma(i_s) = i_s$. Now remove the path between $v_{i_s}$ and $w$ from the tree, and invoke the induction hypothesis on the remaining tree (and the remaining collection of upward paths, which is clearly still edge-disjoint). This completes the induction and thus the proof.
\end{proof}

A $d$-ary tree is a rooted tree for which each internal vertex has precisely $d$ children. For these trees, we find that the characteristic polynomial, evaluated at one specific point, only depends on the number of leaves, but not the tree itself. This is particularly interesting for binary trees ($d=2$), where this value yields the number of edge-disjoint collections of upward paths. A comparable result is the fact that the determinant of the distance matrix of trees only depends on the number of vertices, but not the tree structure (a theorem due to Graham and Pollak \cite{Graham1971addressing}; see also \cite{Edelberg1976distance,Graham1978distance}).

\begin{thm}\label{thm:det_dary}
Let $T$ be a $d$-ary tree with $n$ leaves. We have
$$\det \Big(\frac{1}{d-1} I + C(T) \Big) = (-1)^n \Gamma_T \Big( - \frac{1}{d-1} \Big) = (d-1)^{-n} d^{d (n-1)/(d-1)}.$$
Equivalently, if $\operatorname{int}(T)$ is the number of internal vertices,
$$\det \big( I + (d-1) C(T) \big) = d^{d \operatorname{int}(T)}.$$
In particular, a binary tree with $n$ leaves has $4^{n-1}$ edge-disjoint collections of upward paths, which is independent of the precise shape of the tree.
\end{thm}

\begin{proof}
We prove the statement by induction on the number of internal vertices: it is well known that a $d$-ary tree with $n$ leaves has $\frac{n-1}{d-1}$ internal vertices. If the tree only consists of a single leaf, so that $L(T) = n = 1$ and there are no internal vertices, the formula reduces to $\frac{1}{d-1} = \frac{1}{d-1}$ and is thus readily seen to hold.

For the induction step, consider an internal vertex $v$ whose level is maximal. All its children are leaves, and without loss of generality we can assume that these children correspond to the last $d$ rows of $C(T)$. Thus the matrix $C(T) + \frac{1}{d-1} I $ has the form
$$C(T) + \frac{1}{d-1} I = \begin{bmatrix} & & & & \\ \qquad \mathbf{B} \qquad & \mathbf{a}^t & \mathbf{a}^t & \cdots & \mathbf{a}^t \\ & & & & \\ \mathbf{a} & k+\frac{d}{d-1} & k & \cdots & k \\ \mathbf{a} & k & k+\frac{d}{d-1} & \cdots & k \\ \vdots & \vdots & \vdots & \ddots & \vdots \\
\mathbf{a} & k & k & \cdots & k+\frac{d}{d-1}  \end{bmatrix},$$
where $\mathbf{B}$ is a matrix, $\mathbf{a}$ a row vector, and $k$ is the level of vertex $v$.
If the $d$ leaves are removed, so that $v$ becomes a leaf, the resulting tree $T^\prime$ is again a $d$-ary tree with $L(T^\prime) = L(T) - (d-1)$ (and $\operatorname{int}(T^\prime) = \operatorname{int}(T) - 1$), and we have
$$\frac{1}{d-1} I  + C(T^\prime) = \begin{bmatrix} & \\ \qquad \mathbf{B} \qquad & \mathbf{a}^t \\ & \\ \mathbf{a} & k + \frac{1}{d-1} \end{bmatrix}.$$
We shall prove that
$$\det \Big(\frac{1}{d-1} I + C(T) \Big) = \det \Big(\frac{1}{d-1} I + C(T^\prime) \Big) \cdot \frac{d^d}{(d-1)^{d-1}},$$
so that the desired formula follows directly from the induction hypothesis.

To this end, we subtract the first of the final $d$ rows of $\frac{1}{d-1} I + C(T)$ from the other $d-1$ rows, and then the second of the last $d$ columns from the final $d-2$ columns to obtain
\begin{align*}
\det \Big(\frac{1}{d-1} I + C(T) \Big) &= \begin{vmatrix} & & & & & \\ \qquad \mathbf{B} \qquad & \mathbf{a}^t & \mathbf{a}^t & \mathbf{a}^t & \cdots & \mathbf{a}^t \\ & & & & & \\ \mathbf{a} & k+\frac{d}{d-1} & k & k & \cdots & k \\ \mathbf{a} & k & k+\frac{d}{d-1} & k & \cdots & k \\ \mathbf{a} & k & k & k+\frac{d}{d-1} & \cdots & k \\
\vdots & \vdots & \vdots & \vdots & \ddots & \vdots \\
\mathbf{a} & k & k & k & \cdots & k+\frac{d}{d-1}  \end{vmatrix} \\
&= \begin{vmatrix} & & & & & \\ \qquad \mathbf{B} \qquad  & \mathbf{a}^t & \mathbf{a}^t & \mathbf{a}^t & \cdots & \mathbf{a}^t \\ & & & & & \\ \mathbf{a} & k+\frac{d}{d-1} & k & k & \cdots & k \\ \mathbf{0} & -\frac{d}{d-1} & \frac{d}{d-1} & 0 & \cdots & 0 \\ \mathbf{0} & -\frac{d}{d-1} & 0 & \frac{d}{d-1} & \cdots & 0 \\
\vdots & \vdots & \vdots & \vdots & \ddots & \vdots \\
\mathbf{0} & -\frac{d}{d-1} & 0 & 0 & \cdots & \frac{d}{d-1}  \end{vmatrix} \\
&= \begin{vmatrix} & & & & & \\ \qquad\mathbf{B}\qquad & \mathbf{a}^t & \mathbf{a}^t & \mathbf{0}^t & \cdots & \mathbf{0}^t \\ & & & & & \\ \mathbf{a} & k+\frac{d}{d-1} & k & 0 & \cdots & 0 \\ \mathbf{0} & -\frac{d}{d-1} & \frac{d}{d-1} & -\frac{d}{d-1} & \cdots & -\frac{d}{d-1} \\ \mathbf{0} & -\frac{d}{d-1} & 0 & \frac{d}{d-1} & \cdots & 0 \\
\vdots & \vdots & \vdots & \vdots & \ddots & \vdots \\
\mathbf{0} & -\frac{d}{d-1} & 0 & 0 & \cdots & \frac{d}{d-1}  \end{vmatrix}.
\end{align*}
Now add each of the last $d-2$ rows to the $(d-1)$-th row from the bottom:
$$\det \Big(\frac{1}{d-1} I + C(T) \Big) = \begin{vmatrix} & & & & & \\ \qquad\mathbf{B}\qquad & \mathbf{a}^t & \mathbf{a}^t & \mathbf{0}^t & \cdots & \mathbf{0}^t \\ & & & & & \\ \mathbf{a} & k+\frac{d}{d-1} & k & 0 & \cdots & 0 \\ \mathbf{0} & -d & \frac{d}{d-1} & 0 & \cdots & 0 \\ \mathbf{0} & -\frac{d}{d-1} & 0 & \frac{d}{d-1} & \cdots & 0 \\
\vdots & \vdots & \vdots & \vdots & \ddots & \vdots \\
\mathbf{0} & -\frac{d}{d-1} & 0 & 0 & \cdots & \frac{d}{d-1} \end{vmatrix}.$$
Next, expand the determinant with respect to the last $d-2$ columns, one by one:
$$\det \Big(\frac{1}{d-1} I + C(T) \Big) = \Big( \frac{d}{d-1} \Big)^{d-2} \begin{vmatrix} & & \\ \qquad\mathbf{B}\qquad & \mathbf{a}^t & \mathbf{a}^t \\ & & \\ \mathbf{a} & k+\frac{d}{d-1} & k \\ \mathbf{0} & -d & \frac{d}{d-1} \end{vmatrix}.$$
Finally, subtract the second to last column from the last in the remaining matrix, then add $\frac{d-1}{d}$ times the last column back to the previous column:
\begin{align*}
\det \Big(\frac{1}{d-1} I + C(T) \Big)
&= \Big( \frac{d}{d-1} \Big)^{d-2} \begin{vmatrix} & & \\ \qquad\mathbf{B}\qquad & \mathbf{a}^t & \mathbf{0}^t \\ & & \\ \mathbf{a} & k+\frac{d}{d-1} & -\frac{d}{d-1} \\ \mathbf{0} & -d & \frac{d^2}{d-1} \end{vmatrix} \\
&= \Big( \frac{d}{d-1} \Big)^{d-2} \begin{vmatrix} & & \\ \qquad\mathbf{B}\qquad & \mathbf{a}^t & \mathbf{0}^t \\ & & \\ \mathbf{a} & k+\frac{1}{d-1} & -\frac{d}{d-1} \\ \mathbf{0} & 0 & \frac{d^2}{d-1} \end{vmatrix} \\
&= \frac{d^d}{(d-1)^{d-1}} \begin{vmatrix} & \\ \qquad\mathbf{B}\qquad & \mathbf{a}^t \\ & \\ \mathbf{a} & k + \frac{1}{d-1} \end{vmatrix} \\
&= \det \Big( \frac{1}{d-1} I + C(T^\prime) \Big),
\end{align*}
which completes the induction. To transform
$$\det \Big(\frac{1}{d-1} I + C(T) \Big) = (d-1)^{-n} d^{d (n-1)/(d-1)}$$
into
$$\det \big( I + (d-1) C(T) \big) = d^{d \operatorname{int}(T)},$$
one simply needs to recall that $\operatorname{int}(T) = \frac{n-1}{d-1}$. The special case $d=2$ yields the number of edge-disjoint collections of upward paths by Theorem~\ref{thm:edge-disjoint}.
\end{proof}

\bibliographystyle{abbrv}
\bibliography{AncestralMatrix}

\end{document}